\documentclass[10pt]{amsart}

\usepackage{amssymb,amsthm,amsmath}
\usepackage[numbers,sort&compress]{natbib}
\usepackage{color}
\usepackage{graphicx}
\usepackage{tikz}
\usepackage{amssymb,amsthm,amsmath}
\usepackage[numbers,sort&compress]{natbib}
\usepackage{color}
\usepackage{graphicx}
\usepackage{tikz}

\title[ ]{  Noncompact complete Riemannian manifolds with dense eigenvalues embedded in the
essential spectrum of the Laplacian
 }
\author{ Svetlana Jitomirskaya}
\address[ Svetlana Jitomirskaya]{ Department of Mathematics, University of California, Irvine, California 92697-3875, USA}
\email{szhitomi@math.uci.edu}

\author{Wencai Liu}
\address[Wencai Liu]{Department of Mathematics, University of California, Irvine, California 92697-3875, USA}\email{liuwencai1226@gmail.com}

\theoremstyle{plain}
\newtheorem{theorem}{Theorem}[section]
\newtheorem{corollary}[theorem]{Corollary}

\newcommand{\R}{\mathbb{R}}
\newcommand{\Z}{\mathbb{Z}}
\theoremstyle{definition}

\newtheorem{remark}[theorem]{Remark}

\begin{document}


\begin{abstract}
We prove sharp criteria on the behavior of
radial curvature for the existence of  asymptotically flat or hyperbolic Riemannian manifolds with prescribed sets of
eigenvalues embedded in the spectrum of the Laplacian. In particular,
we construct such manifolds with dense embedded point spectrum and
sharp curvature bounds.

\end{abstract}
\maketitle
\section{Introduction and main results}
Let $(M_n, g)$ be an $n$-dimensional noncompact
complete Riemannian manifold. The Laplace-Beltrami operator
$\Delta=\Delta_g$ on   $(M_n,g)$ is essentially self-adjoint on
$C^{\infty}_0(M_n)$.
We also denote by $\Delta$ its unique self-adjoint extension to
$L^2(M_n,dv_g)$.

For compact $M_n,$ there is a wealth of results on
 the relations between the
 geometry of the manifold and spectral properties of the Laplacian
 (which in this case only has discrete eigenvalues). These relations
 are not so well studied for the nocompact case, which is the subject
 of this paper. We will mention some results later in this article but
 refer the readers to \cite{donn} for a more complete review. The past
 work
 has been mostly focused on proofs of the purity of absolutely
 continuous spectrum and absence of embedded eigenvalues. Here we
 study the opposite question.

If $M$ has constant radial curvature $-K_0,$ then, for negative $K_0,$
$-\Delta$ has a complete set of eigenvalues $\{\lambda_n\geq 0\}$ with
$\lambda_n\to \infty$ \cite{spg}. For $K_0\geq 0$,
$\sigma(-\Delta)=\sigma_{{\rm
    ess}}(-\Delta)=\left[\frac{K_0}{4}(n-1)^2,\infty \right)$  and
there are no eigenvalues \cite{spg}.  For perturbations of the latter case
it is natural to expect that whether there are eigenvalues will depend on the size of the
perturbation. Perturbations on a compact set can only create
eigenvalues below the essential spectrum. Thus the question whether
one can embed eigenvalues in the essential spectrum will depend on the rate
 of approach of $-K_0$ by the radial curvature $K(r)$ of the
 perturbation, at infinity. Manifolds with $K(r)\to -K_0$ as
 $r\to\infty$ are called asymptotically flat if $K_0=0$ and
 asymptotically hyperbolic if $K_0=1.$ The case $K_0>0$ can be
 rescaled to $K_0=1$ but we find it more useful to keep $K_0$ and will
 call all manifolds with $K(r)\to -K_0,\; K_0>0,$ asymptotically hyperbolic.

In this paper we answer the following question. Given any finite
or countable (possibly dense) set $A\subset
\sigma(-\Delta)=\left[\frac{K_0}{4}(n-1)^2,\infty \right),$ can we
construct an asymptotically flat or hyperbolic $n$-dimensional manifold with an
embedded (in the absolutely continuous spectrum) eigenvalue at each $\lambda \in
A?$ How is it influenced by the asymptotical behavior of the radial curvature?

We prove

\begin{theorem}\label{1}
For any countable $A\subset
\left[\frac{K_0}{4}(n-1)^2,\infty \right)$ there exist
asymptotically flat and asymptotically hyperbolic $n$-dimensional Riemannian manifolds  with
$\sigma_{ac}(-\Delta)= \left[\frac{K_0}{4}(n-1)^2,\infty \right) $ and an
embedded eigenvalue at each $\lambda \in
A.$
\end{theorem}

In particular, this of course implies

\begin{corollary}
There exist asymptotically flat and asymptotically hyperbolic
$n$-dimensional  Riemannian manifolds  with dense point spectrum embedded in the absolutely
continuous spectrum.
\end{corollary}

An interesting question is to study the curvature conditions for
possibility to embed an {\it arbitrary} countable set in the
absolutely continuous spectrum. It turns out, for finite sets it can
be done with $r|K(r) +K_0|=O(1),$  while for countable sets it is
enough to require that  $r|K(r) +K_0| \to\infty,$ no matter how
slowly. We have
\begin{theorem}\label{2}
\begin{enumerate}

\item For any finite $A\subset
\left[\frac{K_0}{4}(n-1)^2,\infty \right)$, there exists an
$n$-dimensional manifold with $r|K(r) +K_0|<C,$ $\sigma_{ac}(-\Delta)= \left[\frac{K_0}{4}(n-1)^2,\infty \right), $ and an
embedded eigenvalue at each $\lambda \in
A.$
\item For any countable $A\subset
\left[\frac{K_0}{4}(n-1)^2,\infty \right)$ and any $C(r)>0$ with  $\lim_{r\to\infty} rC(r)=\infty,$ there exists an
$n$-dimensional manifold with $|K(r) +K_0|<C(r),$ $\sigma_{ac}(-\Delta)= \left[\frac{K_0}{4}(n-1)^2,\infty \right), $ and an
embedded eigenvalue at each $\lambda \in
A.$
\end{enumerate}
\end{theorem}

{\bf Remark.}
\begin{enumerate}
\item
Theorem \ref{2} is sharp in the asymptotically hyperbolic
  case in the following sense. If $r|K(r) +K_0|\to 0$ (or even is
  bounded by a sufficiently small constant),  there can be no embedded
  eigenvalues \cite{kumura2010radial}. We conjecture it is also sharp in a similar sense in the
  asymptotically flat case.
\item  We conjecture that, at least in the hyperbolic case, (2) of
  Theorem \ref{2} is sharp in an even stronger sense. Namely, that given a monotone $C(r)>0,$ for any countable $A\subset
 \left[0,\infty \right)$  there exists an
 $n$-dimensional manifold with $|K(r) +K_0|<C(r),$ $\sigma_{ac}(-\Delta)= \left[\frac{K_0}{4}(n-1)^2,\infty \right), $ and an
 embedded eigenvalue at each $\lambda \in
 A$ {\it if and only if}  $\lim_{r\to\infty} rC(r)=\infty.$ Given
 Theorem \ref{2} this statement would only require proving that  if $C(r)$ is bounded,
only eigenvalues below a certain threshold can be embedded, as is the
case for $n=1$ \cite{kumura2010radial}.
\end{enumerate}

Let us present more detail on the history. For the asymptotically hyperbolic case,
the sharp transition on a possibility to embed one eigenvalue was
given by Kumura   \cite{kumura2010radial} based on the arguments of Kato \cite{kato}.
 He   excluded eigenvalues greater than $\frac{K_0(n-1)^2}{4}$ under the assumption that $K_{\rm rad} + K_0 = o(r^{-1})$, and  also  constructed a manifold with the radial curvature  $K_{\rm rad} + K_0 = O(r^{-1})$ and with an eigenvalue $\frac{K_0(n-1)^2}{4} + 1$ embedded  into the essential spectrum $[ \frac{K_0(n-1)^2}{4} , \infty )$.

Previous results on absence of embedded eigenvalues under the radial
curvature conditions are reviewed in \cite{donn}. They go back to Pinsky \cite{pinsky1979spectrum}, with
a later milestone by Donnelly \cite{donnelly1990negative}. 
Some recent results on the absence of   eigenvalues can be found in \cite{liugrowth1,liugrowth2}.

For the asymptotically flat case, the absence of embedded eigenvalues
results go back to \cite{karp}. Kumura, Donnelly, and Garofalo
\cite{kumuraflat,donnelly1992,donnelly1999} showed the absence of
positive eigenvalues of  the Laplacian
if the curvature $K_{\rm rad}   = o( r^{-2} )$. As mentioned, we conjecture here that Laplacian has no positive eigenvalues if $K_{\rm rad}   = o( r^{-1} )$.

Under certain stronger curvature decay assumption on the perturbation, limiting
absorption principle, originally from  Agmon's theory
\cite{agmon1975}, holds for the Laplacian. See
\cite{kumura2013limiting,Tao}  and references therein. In this case,
Laplacian   has purely absolutely continuous spectrum.

This paper is the first one in the series where we construct manifolds
with unusual spectral properties of the Laplacian  and certain sharp curvature
bounds. For example, in the upcoming \cite{jl1} we obtain Riemannian manifolds
with singular continuous spectrum embedded in the spectrum of the
Laplacian.

The Riemannian manifolds $(M,g)$ we construct  are rotationally
symmetric, and we construct rotationally symmetric eigenfunctions, thus
reducing the problem to a one-dimensional Schr\"odinger operator.
Fix some $O\in M$ as the origin. Using the radial coordinates (from
$O$) we construct  Riemannian manifold with the structure of the form $(M,g) =
\bigl( {  \R}^n, dr^2 + f_1^2(r) g_{S^{n-1}(1)} \bigr)$ where $
g_{S^{n-1}(1)} $ is the standard Riemannian metric on the unit
sphere, and we need to construct $f_1$ so that the Laplacian has the desired properties.
Suppose $h(r)$ is a function on $M$ only depending on the radius $r$.
Then the Laplacian is equivalent to the following one-dimensional Schr\"odinger operator,
\begin{align}\label{0lap1}
    - \Delta_{g} \bigl( h(r)  \bigr)
    = -\left\{ \frac{\partial ^2}{\partial r^2}
    + (n-1) S(r) \frac{\partial }{\partial r} \right\} h(r),
\end{align}
where
\begin{equation}\label{Sr}
   S(r)=\frac{f_1^\prime(r)}{f_1(r)}.
\end{equation}
In order to make the  manifold smooth in the neighborhood of $O$,  $ f_1^{\text{even}}(0)$ must vanish at $0$ and $f_1^{\prime}(0)\neq 0 $.
This implies  $S(r)$ is singular at $0$. Thus we need to deal with
one-dimensional Schr\"odinger operator (\ref{0lap1}) with
singularities at both $0$ and $\infty$.

In the neighborhood of $\infty$ by the Liouville transformation,
Laplacian (\ref{0lap1}) can be normalized to a Schr\"odinger operator
of the form \begin{equation}\label{S}-\frac{d^2}{dx^2}+q(x).
\end{equation}

Constructing operators (\ref{S}) with given sets embedded as eigenvalues
in the essential spectrum is an old question, going back to the
celebrated work of Wigner-von Neuman \cite{von1929uber} who
constructed an explicit potential with an embedded (given) eigenvalue.

Simon \cite{simon1997some}  showed that for any rate of decay $h(x)$ that is
slower than Coulomb and
any countable subset $A\subset\R^+,$ there exists  $q(x)$ bounded by
$h(x),$ so that operator (\ref{S}) (whole-line or with the desired boundary
conditions) has  an embedded eigenvalue at each $\lambda\in A.$ However, the potential in Simon's construction is not
continuous and thus cannot be used for our purposes. Previously,
Naboko \cite{naboko1986dense} constructed   smooth  potentials with
the same property but only if elements of $A$ are rationally independent.

Naboko's construction starts from the origin.  Then  he first
constructs piecewise-constant potentials with desired properties,
which are then smoothed out. Simon uses a different method. He uses the
Wigner-von Neumann type to construct the desired potential and  the method of  $L^2$  perturbations
to guarantee  boundary conditions at the origin $0$ for the
eigenfunctions. His construction starts at $\infty$, thus it is
nontrivial to make it smooth.

In this paper we develop a new construction, based on {\it piecewise}
Wigner-von Neumann potentials, different from both \cite{naboko1986dense}  and
\cite{simon1997some}. In fact, we view the construction itself as one
of the important achievements of this paper. It  is robust and fundamental in that it can be applied in a variety
of contexts to construct embedded eigenvalues.
In the forthcoming work it is adapted by one of the authors and Ong to
construct eigenvalues embedded into the spectral band for perturbed
periodic operator, in both continuous and discrete settings
\cite{ld1,liu2018criteria}, and also to construct eigenvalues embedded into the absolutely continuous spectrum for  perturbed Stark type operators \cite{liu,liu2018sharp}.

First, in order to deal with singularities at
both $0$ and infinity,  it is natural to  construct  Riemannian metric around $0$ and $\infty$
separately so that the two operators  $- \frac{\partial ^2}{\partial r^2}
    -(n-1) S(r) \frac{\partial }{\partial r} $ (one  around 0,
    another around $\infty$)  have the given eigenvalues.

 We start at  $O$ with the standard Euclidean
metric for $r\leq \frac{1}{2}$. Thus the eigenfunctions of the Laplacian (\ref{0lap1}) are given by
the Bessel functions.  In the neighborhood of $\infty$ ($r\geq 3$)  we  use piecewise Wigner-von Neumann type functions  to
construct the $f_1(r)$, adding the eigenvalues one (or fewer) at a time. We allow Wigner-von Neumann type potentials to
be adapted in the next segment to balance  the   boundary conditions
of the associated eigenfunctions.

When further eigenvalues are taken into  consideration,  we need to
adapt the next segment of Wigner-von Neumann type potential  to
balance the new boundary conditions. However,  Wigner-von Neumann type
potential (associated to a fixed eigenvalue) may significantly change other eigenfunctions.
 As we add new eigenvalues, the change will accumulate.
 To overcome this difficulty,
we use the quantitative analysis to study the relationships for all  the  Wigner-von Neumann type potentials, corresponding  eigenfunctions
and the other eigenfunctions. In particular, an important building
block is a Theorem that allows to construct a Wigner-von Neumann type
potential on a sufficiently long interval so that a solution for a
given energy with given boundary conditions decays, while solutions for
energies from a given finite set, for {\it all}  boundary conditions, do not grow too much. Then we proceed with double
induction, so that at each new step we add an interval with decay for each
previously  treated energy, to ensure the overall decay.

After the separate construction,
we need to connect the Riemannian metric  at $r<\frac{1}{2}$ and $r>3$  smoothly so that the eigenfunctions of
the two separate operators on  $r<\frac{1}{2}$ and $r>3$   connect
smoothly. This can be done if the boundary conditions of
eigenfunctions match at some  fixed point $r\in[\frac{1}{2},3]$.

As we work with the one-dimensional construction,  the proof of  Theorems
\ref{Thembedfinite} and \ref{Thembedcountable}, establishes also  the following Theorems.

\begin{theorem}\label{Halfline}
Let $\{\lambda_j\}$ be an arbitrary set of distinct positive numbers. Let $\{\theta\}_j$
be a sequence of angles in $[0,\pi]$. If the  set $\{\lambda_j\}$ is finite, then there exist potentials $q (x)\in C^{\infty}[0,\infty)$  such that
\begin{enumerate}
   \item
  for each $j$, $(-D^2+q)u=\lambda_j u$ has an $ L^2(\R^+)$ solution  such that
  \begin{equation*}
    \frac{u^\prime(0)}{u(0)}=\tan\theta_j,
  \end{equation*}

   \item  $|q(x)|= \frac{O(1)}{|x|+1}$.
\end{enumerate}
If the set $\{\lambda_j\}$ is countable, then for any function  $C(x)>0$     on  $(0,\infty)$  with $  \lim_{x\to \infty}C(x) = \infty$,  there exist  potentials
 $q (x)\in C^{\infty}[0,\infty)$  such that
\begin{enumerate}
   \item
  for each $j$, $(-D^2+q)u=\lambda_j u$ has an $ L^2(\R^+)$ solution  such that
  \begin{equation*}
    \frac{u^\prime(0)}{u(0)}=\tan\theta_j,
  \end{equation*}

   \item  $|q(x)|\leq  \frac{C(x)}{|x|+1}$.
\end{enumerate}
\end{theorem}
\begin{theorem}\label{Wholeline}
Let $\{\lambda_j\}$ be an arbitrary set of distinct positive numbers.  If the  set $\{\lambda_j\}$ is finite, then there exist  potentials $q (x)\in C^{\infty}(-\infty,\infty)$  such that
\begin{enumerate}
   \item
  for each $j$, $(-D^2+q)u=\lambda_j u$ has an $ L^2(\R)$ solution.
   \item  $|q(x)|=\frac{O(1)}{|x|+1}$.
\end{enumerate}
If the set $\{\lambda_j\}$ is countable, then for any function  $C(x)>0$     on  $(0,\infty)$  with $  \lim_{x\to \infty}C(x) = \infty$,  there exist potentials $q (x)\in C^{\infty}(-\infty,\infty)$  such that
\begin{enumerate}
   \item
  for each $j$, $(-D^2+q)u=\lambda_j u$ has an $ L^2(\R)$ solution,

   \item  $|q(x)|\leq  \frac{C(x)}{|x|+1} $ for $x\in\R$.
\end{enumerate}
\end{theorem}

 \section{Preparation}
 The following result is well known. See \cite{klaus1991asymptotic} or page 93 in \cite{eastham1982schrodinger}.
\begin{theorem}\label{thasy1}
Let $V(x)$ be a continuous function on $(R_0,\infty)$ of the form  $V(x)=4a\kappa\frac{\sin(2\kappa x+\phi)}{x}+V_1(x)$ with $x>R_0$ and $a\neq 0$, where
$|V_1(x)\leq \frac{\hat{C}}{x^2}$.
 Consider the  differential equation $-y^{''}+Vy=\lambda y$ with $\lambda>0$. Then the following asymptoticis holds (uniform with respect to $\phi$) as $x$ goes to infinity
\begin{enumerate}
   \item
    if $\kappa\neq \pm \sqrt{\lambda}$, then there exists a fundamental system of solutions $\{y_1(x),y_2(x)\}$ such that
    $y_1(x)=\cos\sqrt{\lambda} x+O(\frac{1}{x})$, $y_1^\prime(x)=-\sqrt{\lambda}\sin\sqrt{\lambda} x+O(\frac{1}{x})$,
    and $y_2(x)=\sin\sqrt{\lambda} x+O(\frac{1}{x})$, $y_2^\prime(x)=\sqrt{\lambda}\cos\sqrt{\lambda} x+O(\frac{1}{x})$;
    \item
     if $\kappa= \pm \sqrt{\lambda}$, then there exists a fundamental system of solutions $\{y_1(x),y_2(x)\}$ such that
    $y_1(x)=x^{-a}(\cos(\sqrt{\lambda} x+\frac{\phi}{2})+O(\frac{1}{x}))$, $y_1^\prime(x)=-\sqrt{\lambda} x^{-a}(\sin(\sqrt{\lambda} x+\frac{\phi}{2})+O(\frac{1}{x}))$,
    and $y_2(x)=x^{a}(\sin(\sqrt{\lambda} x+\frac{\phi}{2})+O(\frac{1}{x}))$, $y_2^\prime(x)=\sqrt{\lambda} x^{a}(\cos(\sqrt{\lambda } x+\frac{\phi}{2})+O(\frac{1}{x}))$.
    Moreover,  $y_1(x),y_1^\prime(x),y_2(x),y_2^\prime(x)$ are jointly  continuous with respect to $x,\phi$.
\end{enumerate}

\end{theorem}
The following theorem is an important building block for our inductive
construction. It allows to construct a potential with desired bounds
that ensures decay of the solution  for a given energy/boundary
condition and stabilization (not much growth) of solutions for
energies from a given finite set with arbitrary boundary conditions.
\begin{theorem}\label{Twocase}
Suppose   $\lambda>0$ and $ A=\{\hat{\lambda}_j>0\}_{j=1}^k$ with
$\lambda\notin A.$
Suppose  $\theta_0\in[0,\pi]$. Let $x_1>x_0>b$. For any function $\widetilde {V}$, we define
\begin{equation*}
 \widetilde q(x)=  \frac{(n-1)^2}{4}(\sqrt{K_0}+\widetilde {V}(x))^2+\frac{n-1}{2}\widetilde {V}^{\prime}(x).
\end{equation*}

Then there exist constants $K(E, A,K_0)$, $C(E, A,K_0)$ (independent of $b, x_0$ and $x_1$) and potential $ \widetilde V(x,E,A,x_0,x_1,b,\theta_0)$  such that  for $x_0-b>K(E,A,K_0)$ the following holds:

   \begin{description}
     \item[Curvature]   for $x_0\leq x \leq x_1$, ${\rm supp}( \widetilde {V})\subset(x_0,x_1)$ and $  \widetilde {V}\in C^{\infty}(x_0,x_1)$,  and
     \begin{equation}\label{thm141}
        |  \widetilde {V}(x,E,A,x_0,x_1,b,\theta_0)|\leq \frac{C(E, A,K_0)}{x-b},
     \end{equation}
and
\begin{equation}\label{thm141nov}
        |  \widetilde {V}^\prime  (x,E,A,x_0,x_1,b,\theta_0)|\leq \frac{C(E, A,K_0)}{x-b}.
     \end{equation}
     \item[Solution for $\lambda$]  the  solution   of  $(-D^2+\widetilde q)y_{\lambda}=(\lambda +\frac{(n-1)^2}{4}K_0)y_{\lambda}$  with
     boundary condition $ \frac{y^\prime(x_0)}{y(x_0)}=\tan\theta_0$ satisfies
    \begin{equation}\label{Key2}
      ||\left(\begin{array}{c}
      y_{\lambda}(x_1) \\
      \frac{1}{\sqrt{\lambda}}y_{\lambda}^{\prime}(x_1)
    \end{array}
    \right)||\leq 2(\frac{x_1-b}{x_0-b})^{-100}||\left(\begin{array}{c}
      y_{\lambda}(x_0) \\
     \frac{1}{\sqrt{\lambda}} y_{\lambda}^{\prime}(x_0)
    \end{array}
    \right)||,
    \end{equation}
    and   for $x\in[x_0,x_1]$,
    \begin{equation}\label{Controlallx1}
      ||\left(\begin{array}{c}
      y_{\lambda}(x) \\
      \frac{1}{\sqrt{\lambda}}y_{\lambda}^{\prime}(x)
    \end{array}
    \right)||\leq  2||\left(\begin{array}{c}
      y_{\lambda}(x_0) \\
     \frac{1}{\sqrt{\lambda}} y_{\lambda}^{\prime}(x_0)
    \end{array}
    \right)||.
    \end{equation}
      \item[Solution for $\lambda_j$]    for   any solution     of  $(-D^2+ \widetilde q)y_{\lambda_j}=(\lambda_j+\frac{(n-1)^2}{4}K_0)y_{\lambda_j}$, we have
    \begin{equation}\label{Key1}
     ||\left(\begin{array}{c}
      y_{\lambda_j}(x) \\
      \frac{1}{\sqrt{\lambda_j}}y_{\lambda_j}^{\prime}(x)
    \end{array}
    \right)||\leq 2||\left(\begin{array}{c}
      y_{\lambda_j}(x_0) \\
     \frac{1}{\sqrt{\lambda_j}} y_{\lambda_j}^{\prime}(x_0)
    \end{array}
    \right)||,
    \end{equation}
    for all $x_0\leq x\leq x_1$.
   \end{description}

  \end{theorem}

  \begin{proof} In this proof $C(E,A,K_0), K(E,A,K_0)$ will denote
    constants (possibly different in different equations) that depend
    on $E,A,K_0$ only. In the future, however, $C(E,A,K_0),
    K(E,A,K_0)$ will refer to the specific constants as given in the
    statement of Theorem \ref{Twocase}.

  By shifting the Schr\"odinger equation,
   we can assume $b=0$.
   Let
   \begin{equation*}
   V(x)= C(E,A,K_0)\chi_{[x_0+1,x_1-1]}(x)\frac{\sin(2 \sqrt{\lambda} x+\phi)}{x},
   \end{equation*}
  and define \begin{equation*}
  q(x)=  \frac{(n-1)^2}{4}(\sqrt{K_0}+ {V}(x))^2+\frac{n-1}{2} {V}^{\prime}(x).
\end{equation*}
Direct computation implies, for $x_0+1\leq x\leq x_1-1$,
   \begin{equation}\label{qnov11}
    q(x)=\frac{(n-1)^2}{4}K_0+C(E,A,K_0) \frac{\sin(2 \sqrt{\lambda} x+\phi+\phi^\prime)}{x}+V_1(x),
   \end{equation}
   where  $\phi^\prime\in \R$ depends on $K_0,n,\lambda$ explicitly, and ${\rm supp}(  {V}_1)\subset[x_0+1,x_1-1]$
   \begin{equation*}
   |V_1(x)|\leq \frac{C(E,A,K_0)}{x^2}.
   \end{equation*}

We extend  $V(x)$ smoothly to  $\widetilde V$ for $x_0<x<x_1$,  with
${\rm supp}( \widetilde {V})\subset(x_0,x_1)$ and $  \widetilde {V}\in C^{\infty}(x_0,x_1)$.
so that \eqref{thm141} and \eqref{thm141nov} hold.

Let  \begin{equation*}
  \widetilde{q}(x)=  \frac{(n-1)^2}{4}(\sqrt{K_0}+ \widetilde{V}(x))^2+\frac{n-1}{2} \widetilde{V}^{\prime}(x).
\end{equation*}
By \eqref{qnov11}, we have
\begin{equation*}
  \widetilde{q}(x)=  \frac{(n-1)^2}{4}K_0+C(E,A,K_0) \chi_{[x_0+1,x_1-1]}(x)  \frac{\sin(2 \sqrt{\lambda} x+\phi+\phi^\prime)}{x}+V_1(x)+V_2(x),
\end{equation*}
   where ${\rm supp}(   {V}_2)\subset(x_0,x_0+1]\cup [x_1-1,x_1)$ and
   \begin{equation}\label{V2nov}
    |V_2(x)|\leq \frac{C(E,A,K_0)}{x}.
   \end{equation}


  We first prove the  property of the solution for  $\lambda$.   By (2) of Theorem \ref{thasy1}, for $x_0>K(\lambda,A,K_0)$,   there is a solution $y$ of $(-D^2+ \widetilde{q})y_{\lambda}=(\lambda +\frac{(n-1)^2}{4}K_0)y_{\lambda}$ (we only consider $x_0+1\leq x\leq x_1-1$ so that $q=\widetilde{q}$) such that
  \begin{equation}\label{RM}
      |y _{\lambda}(x)-\frac{\cos(\sqrt{\lambda }x+\frac{\phi}{2})}{x^{100}}|\leq \frac{C(E,A,K_0)}{x^{101}},|y_{\lambda}^\prime(x)+ \sqrt{\lambda} \frac{\sin(\sqrt{\lambda} x+\frac{\phi}{2})}{x^{100}}|\leq \frac{C(E,A,K_0)}{x^{101}},
    \end{equation}
    for $x_0+1\leq x\leq x_1-1$.
    In particular,
    \begin{equation}\label{GR0}
      |y _{\lambda}(x_0+1)-\frac{\cos( \sqrt{\lambda} (x_0+1)+\frac{\phi}{2})}{(x_0+1)^{100}}|\leq \frac{C(E,A,K_0)}{x_0^{101}},|y_{\lambda}^\prime(x_0+1)+ \sqrt{\lambda} \frac{\sin(\sqrt{\lambda} (x_0+1)+\frac{\phi}{2})}{(x_0+1)^{100}}|\leq \frac{C(E,A,K_0)}{x_0^{101}}.
    \end{equation}
    Now let us consider
     $(-D^2+ \widetilde{q})y_{\lambda}=(\lambda +\frac{(n-1)^2}{4}K_0)y_{\lambda}$  for $x_0<x\leq x_0+1$.

      For $x_0<x\leq x_0+1$,
      \begin{equation*}
        \widetilde{q}=\chi_{(x_0,x_0+1]}(x)V_2(x).
      \end{equation*}
      By \eqref{V2nov}, one has
     for $x_0\leq x\leq x_0+1$
  \begin{equation}\label{equnov}
    ||\left(\begin{array}{c}
      y_{\lambda}(x) \\
      y_{\lambda}^{\prime}(x)
    \end{array}
    \right)-\left(\begin{array}{c}
      y_{\lambda}(x_0+1) \\
      y_{\lambda}^{\prime}(x_0+1)
    \end{array}
    \right)||\leq \frac{C(E,A,K_0)}{x_0}||\left(\begin{array}{c}
      y_{\lambda}(x_0+1) \\
      y_{\lambda}^\prime(x_0+1)
    \end{array}
    \right)||.
  \end{equation}
  By \eqref{GR0} and \eqref{equnov}, we have
  \begin{equation}\label{GR0nov}
      |y _{\lambda}(x_0)-\frac{\cos( \sqrt{\lambda} (x_0+1)+\frac{\phi}{2})}{(x_0+1)^{100}}|\leq \frac{C(E,A,K_0)}{x_0^{101}},|y_{\lambda}^\prime(x_0)+ \sqrt{\lambda} \frac{\sin(\sqrt{\lambda} (x_0+1)+\frac{\phi}{2})}{(x_0+1)^{100}}|\leq \frac{C(E,A,K_0)}{x_0^{101}}.
    \end{equation}

    If $x_0$ is large enough,
   the range of  $ \frac{y_{\lambda}^\prime(x_0)}{y_{\lambda}(x_0)}$ is  $\R$   when $\phi$ is varied.  Choose a $\phi$ such that
    $ \frac{y_{\lambda}^\prime(x_0)}{y_{\lambda}(x_0)}=\tan\theta_0$.

  Arguing as in the proof of  \eqref{equnov}, one has
     for $x_1-1\leq x\leq x_1$
  \begin{equation}\label{equ1nov}
    ||\left(\begin{array}{c}
      y_{\lambda}(x) \\
      y_{\lambda}^{\prime}(x)
    \end{array}
    \right)-\left(\begin{array}{c}
      y_{\lambda}(x_1-1) \\
      y_{\lambda}^{\prime}(x_1-1)
    \end{array}
    \right)||\leq \frac{C(E,A,K_0)}{x_1}||\left(\begin{array}{c}
      y_{\lambda}(x_1-1) \\
      y_{\lambda}^\prime(x_1-1)
    \end{array}
    \right)||.
  \end{equation}

  By (\ref{RM}), \eqref{equnov} and \eqref{equ1nov}, we get  that the  solution  (up to a constant) of  $(-D^2+  \widetilde{q})y_{\lambda}=(\lambda +\frac{(n-1)^2}{4}K_0)y_{\lambda}$  with
     boundary condition $ \frac{y_{\lambda}^\prime(x_0)}{y_{\lambda}(x_0)}=\tan\theta_0$   satisfies  
     \begin{equation}\label{Key2nov}
      ||\left(\begin{array}{c}
      y_{\lambda}(x_1) \\
      \frac{1}{\sqrt{\lambda}}y_{\lambda}^{\prime}(x_1)
    \end{array}
    \right)||\leq 2(\frac{x_1}{x_0})^{-100}||\left(\begin{array}{c}
      y_{\lambda}(x_0) \\
     \frac{1}{\sqrt{\lambda}} y_{\lambda}^{\prime}(x_0)
    \end{array}
    \right)||,
    \end{equation}
    and   for $x\in[x_0,x_1]$,
    \begin{equation}\label{Controlallx1nov}
      ||\left(\begin{array}{c}
      y_{\lambda}(x) \\
      \frac{1}{\sqrt{\lambda}}y_{\lambda}^{\prime}(x)
    \end{array}
    \right)||\leq  2||\left(\begin{array}{c}
      y_{\lambda}(x_0) \\
     \frac{1}{\sqrt{\lambda}} y_{\lambda}^{\prime}(x_0)
    \end{array}
    \right)||.
    \end{equation}
Suppose  $x_0\geq K(\lambda,A,K_0).$
By  Theorem \ref{thasy1} again and following the proof of \eqref{Key2nov}, \eqref{Controlallx1nov},  for   any solution     of  $(-D^2+  \widetilde{q})y_{\lambda_j}=(\lambda_j+\frac{(n-1)^2}{4}K_0)y_{\lambda_j}$, we have
    \begin{equation}\label{Key1nov}
     ||\left(\begin{array}{c}
      y_{\lambda_j}(x) \\
      \frac{1}{\sqrt{\lambda_j}}y_{\lambda_j}^{\prime}(x)
    \end{array}
    \right)||\leq 2||\left(\begin{array}{c}
      y_{\lambda_j}(x_0) \\
     \frac{1}{\sqrt{\lambda_j}} y_{\lambda_j}^{\prime}(x_0)
    \end{array}
    \right)||,
    \end{equation}
    for all $x_0\leq x\leq x_1$.

%
  \end{proof}
  Now  consider the Riemannian manifold with  rotationally symmetric   structure
\begin{align*}
     (M, g)
     = \Bigl( \R^+ \times S^{n-1}(1)
     , dr^2 + f_1^{\,2}(r)g_{S^{n-1}(1)} \Bigr).
\end{align*}
Let $dg_0$ be  the
standard measure on the unit sphere $(S^{n-1}(1), g_0)$.
Assume  that  in the neighborhood of the origin $(M,g)$  is the usual Euclidean space with its standard metric $g_0$, i.e.,
$f_1(r)=r$ and $S(r)=\frac{1}{r}$ for $0<r<\frac{1}{2}$. Then we have
  \begin{theorem}\label{Thbes}
   For any  positive number $\lambda>0$, there exists a rotationally invariant function $h_{1,\lambda}(r)$ such that
   \begin{equation}\label{Gbes}
   - \Delta_{g} \bigl( h_{1,\lambda}(r)  \bigr)
    = -  h_{1,\lambda}^{\prime\prime}(r)
    -\frac{n-1}{r}  h_{1,\lambda}^{\prime}(r) =( K_0(n-1)^2/4+\lambda )h_{1,\lambda}(r)
 \end{equation}
 and
   $h_{1,\lambda}\in L^2(M_1,dv_g)$, where $(M_1, g)
     = \Bigl( (0,\frac{1}{2}] \times S^{n-1}(1)
     , dr^2 + r^{\,2}g_{S^{n-1}(1)} \Bigr)$. 
  \end{theorem}
  \begin{proof}
  Let us consider the ODE,
  \begin{equation}\label{Gapr410}
    u^{\prime\prime}+
    \frac{n-1}{r} u^{\prime} +( K_0(n-1)^2/4+\lambda )u=0.
  \end{equation}
  By the Frobenius method,  \eqref{Gapr410} has  a power series solution of the form,
  \begin{equation*}
    u(r)=\sum_{j=0}^{\infty} c_jr^{j+s},
  \end{equation*}
  and $s$ satisfies
  \begin{equation*}
    s(s-1)+(n-1)s=0.
  \end{equation*}
  It implies we can let $s=0$ in \eqref{Gapr410}.
  Thus \eqref{Gapr410} has a solution of the form
  \begin{equation}\label{Gapr411}
   u(r)=\sum_{j=0}^{\infty} c_jr^{j}.
  \end{equation}
  By the definition $dv_g=r^{n-1}dr dg_{0}$ for $0\leq r<\frac{1}{2}$.
  This implies the solution given by \eqref{Gapr411} is in $L^2(M_1,dv_g)$.
  \end{proof}
  \section{Inductive Construction}
We reformulate Theorem \ref{2}  in a more convenient way.

\begin{theorem}\label{Thembedfinite}
 Suppose $K_0\geq 0.$
 Let  $\{\lambda_j>0\}$ be a finite set of distinct  numbers. There exists  a  rotationally symmetric Riemannian manifold
 $(M_n,g) = \bigl( {  \R}^n, dr^2 + f_1^2(r) g_{S^{n-1}(1)} \bigr)$ such that the following holds,
\begin{enumerate}
   \item
   $\sigma_{{\rm ess}}(-\Delta_g)=\sigma_{{\rm ac}}(-\Delta_g)=\left[\frac{K_0}{4}(n-1)^2,\infty \right)$,
   \item$\left\{  \frac{K_0}{4}(n-1)^2  + \lambda_j \right\} \subset \sigma_{\rm p} ( -\Delta_g ) \cap \left(\frac{K_0}{4}(n-1)^2 , \infty \right)  $,
   \item $ K_{{\rm rad}}(r)+K_0=O(r^{-1})$ as $r\to \infty$.
\end{enumerate}
 \end{theorem}
 \begin{theorem}\label{Thembedcountable}
 Suppose $K_0\geq 0$.
Let  $\{\lambda_j>0\}$ be a  countable set of    distinct  numbers. Let $C(r)>0$ be   any 
function  on  $(0,\infty)$  with $  \lim_{r\to \infty}C(r) = \infty$. Then  there exists  a  rotationally symmetric  Riemannian manifold
 $(M_n,g) = \bigl( {  \R}^n, dr^2 + f_1^2(r) g_{S^{n-1}(1)} \bigr)$ such that the following holds,
\begin{enumerate}
   \item
   $\sigma_{{\rm ess}}(-\Delta_g)=\sigma_{{\rm ac}}(-\Delta_g)=\left[\frac{K_0}{4}(n-1)^2,\infty \right)$,
   \item$\left\{  \frac{K_0}{4}(n-1)^2 + \lambda_j \right\} \subset \sigma_{\rm p} ( -\Delta_g ) \cap \left( \frac{K_0}{4}(n-1)^2 , \infty \right)  $,
   \item $ |K_{{\rm rad}}(r)+K_0|\leq \frac{C(r)}{r}$ as $r\to \infty$.
\end{enumerate}
 \end{theorem}

  We will first construct $f_1$ near the origin, and then extend it
  inductively, adding one segment at a time.

For $r\leq \frac{1}{2}$, let
  \begin{equation}\label{Gf10}
    f_1(r)= r.
  \end{equation}
   For $r\in[1,3]$, let
  \begin{equation}\label{Gf10}
    f_1(r)= e^{\sqrt{K_0}(r-1)}.
  \end{equation}
  We extend $f_1(r)$ to $(0,3)$  so that $f_1(r)>0$  and $f_1\in C^{\infty}(0,3]$. Suppose $K_0\geq 0$.

  For any function $f(x)$ on $[0,\infty]$ such that $\text{supp} f\subset (3,\infty)$ and $f\in C^{\infty} (3,\infty)$, if
  for $r\geq 1$,  we let
  \begin{equation}\label{Gf11}
    f_1(r)=\exp(\int_1^r (\sqrt{K_0}+f(x))dx),
  \end{equation}
  then
  $f_1\in C^{\infty}(0,\infty)$ and  for $r\in[1,3]$ \eqref{Gf10} holds.

  Our objective is to construct   $f(x)$ so that Riemannian manifold $(M, g)
     = \Bigl( \R^+ \times S^{n-1}(1)\cup \{O\}
     , dr^2 + f_1^{\,2}(r)g_{S^{n-1}(1)} \Bigr)$ satisfies Theorem \ref{Thembedfinite} or Theorem \ref{Thembedcountable}.

  Set for $r\geq 1$,
     \begin{align}
     S(r)   := & \frac{f_1'(r)}{f_1(r)}=\sqrt{K_0}+f(r),\label{DeSr}\\
     K(r)   := & -\frac{f_1''(r)}{f_1(r)}=-(\sqrt{K_0}+f(r))^2-f^{\prime}(r)\nonumber\\
     = &- {K_0}-2\sqrt{K_0}f(r)-f^2(r)-f^{\prime}(r),\label{DeKr}\\
     q(r) := & \frac{(n-1)(n-3)}{4}S(r)^2 - \frac{(n-1)}{2}K(r)=\frac{(n-1)^2}{4}(\sqrt{K_0}+f(r))^2+\frac{n-1}{2}f^{\prime}(r).\label{Deq_0r}
\end{align}
Thus we have
\begin{eqnarray}
   S(r) -\sqrt{K_0} &=& f(r)\label{Gapr91} \\
   K(r) -K_0 &=&  -2\sqrt{K_0}f(r)-f^2(r)-f^{\prime}(r) \label{Gapr92}\\
  q(r)-\frac{(n-1)^2}{4}K_0 &=&  \frac{(n-1)^2}{2}\sqrt{K_0}f(r)+\frac{(n-1)^2}{4}f^2(r)+\frac{n-1}{2}f^{\prime}(r).\label{Gapr93}
\end{eqnarray}
\begin{remark}\label{K01}
Actually  $K(r)$ is the radial curvature and $\Delta r=(n-1)S(r)$.
\end{remark}

Let  $h_{1,\lambda}$ be given by Theorem \ref{Thbes} on $(0,\frac{1}{2}]$ and we extend it to $[0,1]$ by solving
\begin{align*}
    - \Delta_{g} \bigl( h_{1,\lambda}(r)  \bigr)
    = -\left\{ \frac{\partial ^2}{\partial r^2}
    + (n-1)\frac{f_1'(r)}{f_1(r)} \frac{\partial }{\partial r} \right\} h_{1,\lambda}(r)
    =  \left( \frac{(n-1)^2}{4}{K_0} + \lambda \right) h_{1,\lambda}(r).
\end{align*}

 Suppose there exists a nontrivial solution $w _{\lambda}(x) \in L^2([1,\infty),dx)$ to the equation
\begin{align}\label{eigene}
     \left( -\frac{d^2}{dx^2} + q(x) - \frac{(n-1)^2}{4} K_0\right) w_{\lambda} (x)
     =\lambda w_{\lambda} (x),
\end{align}
with boundary condition
\begin{equation}\label{1boundary}
    \frac{w_{\lambda}^\prime(1)}{w_{\lambda}(1)}= \frac{h_{1,\lambda}^\prime(1)}{h_{1,\lambda}(1)}+\frac{n-1}{2}\sqrt{K_0}.
\end{equation}
Using this function $w_{\lambda}$, we define a function $h_{2,\lambda}$ by
\begin{align}\label{Gapr97}
   h_{2,\lambda} := f_1^{-\frac{n-1}{2}}w_{\lambda} .
\end{align}
It is easy to verify that
\begin{equation}\label{1boundaryma}
    \frac{h_{2,\lambda}^\prime(1)}{h_{2,\lambda}(1)}=\frac{h_{1,\lambda}^\prime(1)}{h_{1,\lambda}(1)}.
\end{equation}
A direct computation shows that the function $h_{2,\kappa} \bigl( r \bigr)$ for $r\geq 1$ satisfies the eigenvalue equation on $(M,g)$:
\begin{align*}
    - \Delta_{g} \bigl( h_{2,\lambda}(r)  \bigr)
    = -\left\{ \frac{\partial ^2}{\partial r^2}
    + (n-1) S(r) \frac{\partial }{\partial r} \right\} h_{2,\lambda}(r)
    =  \left( \frac{(n-1)^2}{4}{K_0} + \lambda \right) h_{2,\lambda}(r)
\end{align*}
and $h_{2,\lambda}(r)\in L^2(M,dv_{g})$, where $dv_g=f_1^{n-1}(r)drdg_0$ for $r\geq 1$.

Define $h_{ \lambda}(r)=h_{1,\lambda}(r)$ for $r\leq 1$ and $h_{\lambda}(r)=\frac{h_{1,\lambda}(1)}{h_{2,\lambda}(1)}h_{2,\lambda}(r)$ for $r\geq 1$.
Combining with (\ref{1boundaryma}), we have
for all $r>0$
\begin{align}\label{eigenor}
    - \Delta_{g} \bigl( h_{\lambda}(r)  \bigr)
    =  \left( \frac{(n-1)^2}{4} {K_0}+ \lambda \right) h_{\lambda}(r).
\end{align}
Thus $\frac{(n-1)^2}{4}{K_0}+\lambda$ is an eigenvalue and $h_{\lambda}$ is the corresponding eigenfunction.

Now given a set   $\{\lambda_j>0\}$, we will construct $f(x)$ piecewise   of  the form as  in Theorem \ref{Twocase}, such that for any $j$,
there exists eigenfunction $w_{\lambda_j}(x) \in L^2([1,\infty),dx)$ that solves equation  (\ref{eigene}) with $\lambda=\lambda_j$ and satisfies the boundary condition
(\ref{1boundary}). By (\ref{eigenor}), $\{\frac{(n-1)^2}{4}{K_0}+\lambda_j\}$ are the eigenvalues of Laplacian $\Delta_{g}$.


Let $N(k)\in \Z^+$ be a non-decreasing  function  on $\Z^+$, $N(1)=1$
($N(k)$ that we choose will be growing  very slowly).
Let $A_k= \{\lambda_1,\lambda_2,\cdots,\lambda_{N(k)}\}$.
For $k\geq 1$,
let
\begin{equation}\label{DeCKk}
 K_{k}=10+\max_{1\leq j\leq N(k)}K(\lambda_j,A_k\backslash \{\lambda_j\},K_0),
\end{equation}
and
\begin{equation}\label{DeCk}
 C_{k}=10+4^{N(k)}+ N(k)^{100}+K_k+\max_{1\leq j\leq N(k)}C(\lambda_j,A_k\backslash \{\lambda_j\},K_0),
\end{equation}
where the $C$ in \eqref{DeCk} and $K$ in\eqref{DeCKk} are given by Theorem \ref{Twocase}.

Define   $T_0=1$  and   $T_{k}=T_{k-1}C_{k}$. Let $J_k=\sum_{i}^kN(i) T_i $.
Then we have
\begin{equation}\label{cnnov}
    C_k\geq 4^{N(k)}, C_k\geq N(k)^{100},
\end{equation}
and
\begin{equation*}
    T_k\geq 10^k, T_k\geq K_k.
\end{equation*}

We can assure that $C_k  $  goes  to infinity arbitrarily slowly if we
choose appropriately slowly growing $N(k).$ We choose $N(k)$ to be the
largest integer such that
\begin{equation*}
    C_k\leq C\ln k,
\end{equation*}
and
\begin{equation}\label{Gapr95}
 2 C_{k+1}^4\leq C\min _{x\in [J_k,J_{k+1}]}\min\{C(x), \ln x\},
\end{equation}
where $C(x)$ is given by Theorem \ref{Thembedcountable} and $C=C(\lambda_1)$.
We then have $N(k)=N$ for sufficiently large $k$ in the construction
of Theorem \ref{Thembedfinite} and  $\lim_{k}N(k)=\infty$ in the construction of Theorem \ref{Thembedcountable}. 


We will also define function $f(x)$ ($\text{supp} f\subset (3,\infty)$ ) and $w_{\lambda_j}$(x), $j=1,2,\cdots $ on $(1,J_k)$ by induction, so that
\begin{enumerate}
\item
$w_{\lambda_j}(x)$ solves
\begin{align}\label{eigenengj}
     \left( -\frac{d^2}{dx^2} + q(x) - \frac{(n-1)^2}{4} \right) w _{\lambda_j}(x)
     =\lambda w_{\lambda_j} (x),
\end{align}
  for $x\in (1,J_k)$  where $q(x)$ on $(1,J_k)$ is given by (\ref{Deq_0r}),
 and satisfies  boundary condition
\begin{equation}\label{1boundaryn}
    \frac{w_{\lambda_j}^\prime(1)}{w_{\lambda_j}(1)}=\frac{h_{1,\lambda_j}^\prime(1)}{h_{1,\lambda_j}(1)}+\frac{n-1}{2}\sqrt{K_0},
\end{equation}
where $h_{1,\lambda_j}$ is given by (\ref{Gbes}).
\item
$w_{\lambda_i}(x)$  for $i=1,2,\cdots, N(k)$ and $k\geq 2$, satisfies
\begin{eqnarray}
 ||\left(\begin{array}{c}
      w_{\lambda_i}(J_{k}) \\
      \frac{1}{\sqrt{\lambda_i}}w_{\lambda_i}^{\prime}(J_{k})
    \end{array}
    \right)||
  &\leq& 2^{N(k)+1}  N(k)^{100} C_{k}^{-99}||\left(\begin{array}{c}
      w_{\lambda_i}(J_{k-1}) \\
     \frac{1}{\sqrt{\lambda_i}} w_{\lambda_i}^{\prime}(J_{k-1}) \end{array}
    \right)|| .\label{eigenj}
\end{eqnarray}

\item  ${\rm supp} f \subset (J_{k-1},J_k)$ and $f\in C^{\infty}(J_{k-1},J_k)$, and
\begin{align}
   | S(x) -\sqrt{K_0}|  \leq  & 2\frac{N(k)C^2_k}{x},\label{controlsr}\\
   | K(x) + K_0| \leq & (4\sqrt{K_0}+8)\frac{N(k)C^2_k}{x},\label{controlkr}
\end{align}
for $x\in (3,J_k)$,
where  $S(x)$ and $K(x)$ are given by \eqref{DeSr}, \eqref{DeKr} respectively.
\end{enumerate}

By our construction, one has
\begin{eqnarray}
  \frac{J_k}{T_{k+1}} &= &  \frac{\sum_{i}^kN(i)T_i}{T_{k+1}} \\
  &\leq &  \frac{N(k) }{C_{k+1}}\sum_{i=1}^k\frac{T_i}{T_k}\\
  &\leq & 2\frac{N(k) }{C_{k+1}}.\label{Tk1Jk}
\end{eqnarray}

Step 1: Let $f(x)=0$ on $(2,J_1)$. Then $q(x)$ given by (\ref{Deq_0r}) is well defined on $(2,J_1)$.

 Let  $w_{\lambda_j} (x) $, $j=1,2,\cdots$  for $x\in(1,J_1]$ be solutions of  the equation
\begin{align}\label{eigenengg}
     \left( -\frac{d^2}{dx^2} + q(x) - \frac{(n-1)^2}{4} K_0\right) w _{\lambda_j}(x)
     =\lambda_j w_{\lambda_j} (x),
\end{align}
with boundary condition
\begin{equation}
    \frac{w_{\lambda_j}^\prime(1)}{w_{\lambda_j}(1)}=\frac{h_{1,\lambda_j}^\prime(1)}{h_{1,\lambda_j}(1)}+\frac{n-1}{2}\sqrt{K_0},
\end{equation}
where $h_{1,\lambda_j}$ is given by (\ref{Gbes}). 

Step $k+1,$ for $k\geq 1$:

 Suppose  we completed the construction  of $f(x)$ for step $k$. That
 is we have defined $f(x),w_{\lambda_j}(x),q(x)$ on $(1,J_k)$.

 Denote $B_{k+1}=\{\lambda_i\}_{i=1}^{N(k+1)}$.
Applying  Theorem  \ref{Twocase} to $x_0=J_k$, $x_1=J_k+T_{k+1}$, $b=0$, $\lambda=\lambda_1$, $\tan\theta_0=\frac{w_{\lambda_1}^\prime(J_k)}{w_{\lambda_1}(J_k)}$ and $A= B_{k+1}\backslash \{\lambda_1\}$,  we can define
$f(x)=\widetilde{V}(x,\lambda_1,B_{k+1}\backslash \{\lambda_1\},J_k,J_{k}+T_{k+1},0,\theta_0)$  for $x\in (J_k, J_k+T_{k+1}]$.
Thus we can define $w_{\lambda_j}(x) $ on $(0,J_{k}+T_{k+1})$  for all possible $j$.
Since the boundary condition of $w_{\lambda_1}(x)$ matches at the point $ J_k$ (guaranteed by $\tan\theta_0=\frac{w_{\lambda_1}^\prime(J_k)}{w_{\lambda_1}(J_k)}$) and  by Theorem  \ref{Twocase}, one has

 \begin{enumerate}
\item
$w_{\lambda_1}(x)$   solves
\begin{align}\label{eigenen2}
     \left( -\frac{d^2}{dx^2} + q(x) - \frac{(n-1)^2}{4} K_0\right) w _{\lambda_1}(x)
     =\lambda_1 w_{\lambda_1} (x),
\end{align}
 for $x\in (1,J_k+T_{k+1})$,
and satisfies the boundary condition   $\frac{w_{\lambda_1}^\prime(1 )}{w_{\lambda_1}(1)}=\frac{h_{1,\lambda_1}^\prime(1)}{h_{1,\lambda_1}(1)}+\frac{n-1}{2}\sqrt{K_0}$.
\item
$w_{\lambda_1}(x)$ satisfies
\begin{eqnarray}
 ||\left(\begin{array}{c}
      w_{\lambda_1}(J_k+T_{k+1}) \\
      \frac{1}{\sqrt{\lambda_1}}w_{\lambda_1}^{\prime}(J_k+T_{k+1})
    \end{array}
    \right)|| &\leq& 2(\frac{J_k+T_{k+1}}{J_k})^{-100}||\left(\begin{array}{c}
      w_{\lambda_1}(J_k) \\
     \frac{1}{\sqrt{\lambda_1}} w_{\lambda_1}^{\prime}(J_k) \end{array}
    \right)||\nonumber \\
  &\leq& 2^{101} N(k)^{100} C_{k+1}^{-100}||\left(\begin{array}{c}
      w_{\lambda_1}(J_k) \\
     \frac{1}{\sqrt{\lambda_1}} w_{\lambda_1}^{\prime}(J_k) \end{array}
    \right)||\nonumber\\
     &\leq&  N(k)^{100} C_{k+1}^{-99}||\left(\begin{array}{c}
      w_{\lambda_1}(J_k) \\
     \frac{1}{\sqrt{\lambda_1}} w_{\lambda_1}^{\prime}(J_k) \end{array}
    \right)||\label{Key3}
\end{eqnarray}
where the second inequality holds by \eqref{Tk1Jk}.
At the same time, by \eqref{Key1}, the solutions for  $\left( -\frac{d^2}{dx^2} + q(x) - \frac{(n-1)^2}{4} K_0\right) w _{\lambda_j}(x)
     =\lambda_j w_{\lambda_j} (x)$, $j=2,3,\cdots,N(k+1)$, satisfy,
     \begin{equation}\label{Gotherapr9}
      ||\left(\begin{array}{c}
      w_{\lambda_1}(J_k+T_{k+1}) \\
      \frac{1}{\sqrt{\lambda_1}}w_{\lambda_1}^{\prime}(J_k+T_{k+1})
    \end{array}
    \right)|| \leq 2||\left(\begin{array}{c}
      w_{\lambda_1}(J_k) \\
     \frac{1}{\sqrt{\lambda_1}} w_{\lambda_1}^{\prime}(J_k) \end{array}
    \right)||.
     \end{equation}

\end{enumerate}

 Suppose we have defined $f(x)$   on $(0,J_k+tT_{k+1}]$ for $t\leq N(k+1)-1$. Let us give the definition on $(0,J_k+(t+1)T_{k+1}]$.

 Applying  Theorem  \ref{Twocase}  to $x_0=J_k+tT_{k+1}$, $x_1=J_k+(t+1)T_{k+1}$, $b=tT_{k+1}$, $\lambda=\lambda_{t+1}$, $A=B_{k+1}\backslash \lambda_{t+1}$ and $\tan \theta_0=\frac{w_{\lambda_{t+1}}^\prime(J_k+tT_{k+1})}{w_{\lambda_{t+1}}(J_k+tT_{k+1})}$,  we can define
 $f(x)=\widetilde{V}(x,\lambda_{t+1}, B_{k+1}\backslash \lambda_{t+1}, J_k+tT_{k+1},J_k+(t+1)T_{k+1},tT_{k+1},\theta_0)$ on $x\in (J_k+tT_{k+1}, J_k+(t+1)T_{k+1})$. Thus we can define $w_{\lambda_j}(x) $ on $(1,J_{k}+(t+1)T_{k+1}]$  for all possible $j$.

Since the boundary condition of $w_{\lambda_{t+1}}(x)$ matches at the point $ J_k+tT_{k+1}$ (guaranteed by $\tan\theta_0=\frac{w_{\lambda_{t+1}}^\prime(J_k+tT_{k+1})}{w_{\lambda_{t+1}}(J_k+tT_{k+1})}$) and  by Theorem  \ref{Twocase}, one has

 \begin{enumerate}
\item
$w_{\lambda_{t+1}}(x)$     solves
\begin{align}\label{eigenen2}
     \left( -\frac{d^2}{dx^2} + q(x) - \frac{(n-1)^2}{4} K_0\right) w _{\lambda_{t+1}}(x)
     =\lambda_{t+1} w_{\lambda_{t+1}} (x),
\end{align}
for $x\in (1,J_k+(t+1)T_{k+1})$,
and satisfies the boundary condition   $\frac{w_{\lambda_{t+1}}^\prime(1 )}{w_{\lambda_{t+1}}(1)}=\frac{h_{1,\lambda_{t+1}}^\prime(1)}{h_{1,\lambda_{t+1}}(1)}+\frac{n-1}{2}\sqrt{K_0}$.
\item
$w_{\lambda_{t+1}}(x)$ satisfies
\begin{eqnarray}
 ||\left(\begin{array}{c}
      w_{\lambda_{t+1}}(J_k+(t+1)T_{k+1}) \\
      \frac{1}{\sqrt{\lambda_{t+1}}}w_{\lambda_{t+1}}^{\prime}(J_k+(t+1)T_{k+1})
    \end{array}
    \right)|| &\leq& 2(\frac{J_k+T_{k+1}}{J_k})^{-100}||\left(\begin{array}{c}
      w_{\lambda_{t+1}}(J_k+tT_{k+1}) \\
     \frac{1}{\sqrt{\lambda_{t+1}}} w_{\lambda_{t+1}}^{\prime}(J_k+tT_{k+1}) \end{array}
    \right)|| \nonumber\\
  &\leq&    N(k)^{100} C_{k+1}^{-99}||\left(\begin{array}{c}
      w_{\lambda_{t+1}}(J_k+tT_{k+1}) \\
     \frac{1}{\sqrt{\lambda_{t+1}}} w_{\lambda_{t+1}}^{\prime}(J_k+tT_{k+1}) \end{array}
    \right)||.\label{Key3}
\end{eqnarray}
At the same time, by \eqref{Key1}, the solutions for  $\left( -\frac{d^2}{dx^2} + q(x) - \frac{(n-1)^2}{4} K_0\right) w _{\lambda_j}(x)
     =\lambda_j w_{\lambda_j} (x)$, $j=1,2,3,\cdots,N(k+1)$ and $j\neq t+1$, satisfy,
     \begin{equation}\label{Gotherapr9}
      ||\left(\begin{array}{c}
      w_{\lambda_j}(J_k+(t+1)T_{k+1}) \\
      \frac{1}{\sqrt{\lambda_j}}w_{\lambda_j}^{\prime}(J_k+(t+1)T_{k+1})
    \end{array}
    \right)|| \leq 2||\left(\begin{array}{c}
      w_{\lambda_j}(tJ_k) \\
     \frac{1}{\sqrt{\lambda_j}} w_{\lambda_j}^{\prime}(tJ_k) \end{array}
    \right)||.
     \end{equation}
\end{enumerate}

 Thus we have defined $f(x)$ by
 induction in $t$  on each $(1, J_k+tT_{k+1})$ and therefore on $(1,J_k+N(k+1)T_{k+1})=(1,J_{k+1})$.

 Let us mention that for $x\in[J_k+tT_{k+1},J_k+(t+1)T_{k+1}]$ and $0\leq t\leq N(k+1)-1$,
 \begin{equation}\label{Consv}
    f(x)= \widetilde V\left(x,\lambda_{t+1}, B_{k+1}\backslash \{\lambda_{t+1}\},J_k+tT_{k+1},J_k+(t+1)T_{k+1}, tT_{k+1},\frac{w_{\lambda_{t+1}}^\prime(J_k+tT_{k+1})}{w_{\lambda_{t+1}}(J_k+tT_{k+1})}\right),
 \end{equation}
 where $\widetilde V$ is taken from Theorem \ref{Twocase}.

 Now we should show that our definition satisfies the $k+1$ step conditions \eqref{eigenengj}-\eqref{controlkr}.

 Let us consider  $||\left(\begin{array}{c}
      w_{\lambda_i}(x) \\
      \frac{1}{\sqrt{\lambda_i}}w_{\lambda_i}^{\prime}(x)
    \end{array}
    \right)|| $ for $i=1,2,\cdots,N(k+1)$.
$||\left(\begin{array}{c}
      w_{\lambda_i}(x) \\
      \frac{1}{\sqrt{\lambda_i}}w_{\lambda_i}^{\prime}(x)
    \end{array}
    \right)|| $ decreases  from   point  $J_k+(i-1)T_{k+1}$ to $J_k+iT_{k+1}$, $i=1,2,\cdots,N(k+1)$, and
  may increase from any point $J_k+(m-1)T_{k+1}$ to $J_k+mT_{k+1}$, $m=1,2,\cdots,N(k+1)$ and $m\neq i$, but no more than by a factor of 2.
  That is
  \begin{equation*}
   ||\left(\begin{array}{c}
      w_{\lambda_i}(J_k+iT_{k+1}) \\
      \frac{1}{\sqrt{\lambda_i}}w_{\lambda_i}^{\prime}(J_k+iT_{k+1})
    \end{array}
    \right)||\leq    N(k)^{100} C_{k+1}^{-99} ||\left(\begin{array}{c}
      w_{\lambda_i}(J_k+(i-1)T_{k+1}) \\
      \frac{1}{\sqrt{\lambda_i}}w_{\lambda_i}^{\prime}(J_k+(i-1)T_{k+1})
    \end{array}
    \right)|| ,
  \end{equation*}
  and  for $m\neq i$,
  \begin{equation*}
   ||\left(\begin{array}{c}
      w_{\lambda_i}(J_k+mT_{k+1}) \\
      \frac{1}{\sqrt{\lambda_i}}w_{\lambda_i}^{\prime}(J_k+mT_{k+1})
    \end{array}
    \right)||\leq 2 ||\left(\begin{array}{c}
      w_{\lambda_i}(J_k+(m-1)T_{k+1}) \\
      \frac{1}{\sqrt{\lambda_i}}w_{\lambda_i}^{\prime}(J_k+(m-1)T_{k+1})
    \end{array}
    \right)||
  \end{equation*}
  by  Theorem  \ref{Twocase}.

   Thus for $i=1,2,\cdots,N(k+1)$,
   \begin{equation*}
    ||\left(\begin{array}{c}
      w_{\lambda_i}(J_{k+1}) \\
      \frac{1}{\sqrt{\lambda_i}}w_{\lambda_i}^{\prime}(J_{k+1})
    \end{array}
    \right)|| \leq 2^{N(k+1)+1} N(k)^{100} C_{k+1}^{-99} ||\left(\begin{array}{c}
      w_{\lambda_i}({J_k}) \\
      \frac{1}{\sqrt{\lambda_i}}w_{\lambda_i}^{\prime}(J_k)
    \end{array}
    \right)|| .
   \end{equation*}

This implies (\ref{eigenj}) for $k+1$.
By the construction of $f(x)$ \eqref{Consv}, \eqref{thm141},  and \eqref{thm141nov} we have
 for $x\in[J_k+tT_{k+1},J_k+(t+1)T_{k+1}]$ and $0\leq t\leq N(k+1)-1$,
 \begin{eqnarray}
  |f^{\prime}(x)|, | f(x)| &\leq& \frac{C_{k+1}}{x-tT_{k+1}} \nonumber\\
    &\leq&  2 \frac{N(k+1)C^2_{k+1}}{x+1}.\label{Gapr94}
 \end{eqnarray}
 This implies \eqref{controlsr} and \eqref{controlkr} by \eqref{Gapr91} and \eqref{Gapr92}.

\section{Proof of Theorems \ref{Thembedfinite} and \ref{Thembedcountable}}
\begin{proof}
Fix a set $\{\lambda_j>0\}$.  We construct $f(x)$ such that
(\ref{eigenengj}), (\ref{1boundaryn}), (\ref{eigenj}), (\ref{controlsr}) and (\ref{controlkr}) hold for any $k$.
It is known that
\begin{align}
 & \Delta r = (n-1)S(r),  \\
  & K_{{\rm rad}} = K(r) .
\end{align}
By our construction, one has
\begin{equation*}
    \lim_{r\to\infty} \Delta r = \lim_{r\to\infty}(n-1)S(r)=(n-1)\sqrt{K_0}.
\end{equation*}
Thus by Theorem 1.2 in \cite{kumura1997essential},
we have
$$\sigma_{{\rm ess}}(-\Delta_g)=\left[\frac{(n-1)^2}{4}K_0,\infty \right).$$
Thus
\begin{equation*}
    \sigma_{{\rm ac}}(-\Delta_g)\subset [\frac{(n-1)^2}{4}K_0,\infty  ).
\end{equation*}
By  \eqref{Gapr93}  and \eqref{Gapr94}, one has for $r\in[J_{k-1},J_k]$,
\begin{equation*}
   | q(r)-\frac{(n-1)^2}{4}K_0|\leq O(1)\frac{N(k)C^2(k)}{r}.
\end{equation*}
By \eqref{Gapr95}, one  has for $r\in[J_{k-1},J_k]$,
\begin{equation*}
 \frac{N(k)C^2(k)}{r} = O(1) \frac{C_k^3}{r}= O(1)\frac{\ln r}{r}.
\end{equation*}
Thus
\begin{equation*}
  | q(r)-\frac{(n-1)^2}{4}K_0|= O(1)\frac{\ln r}{r}.
\end{equation*}
It implies (e.g. \cite{kis1,kis2}), for  the Schr\"odinger operator, $\sigma_{ac}(-D^2+q)=[\frac{(n-1)^2}{4} K_0,\infty  )$, and then
\begin{equation*}
    [\frac{(n-1)^2}{4} K_0,\infty  )\subset \sigma_{{\rm ac}}(-\Delta_g).
\end{equation*}
Thus
\begin{equation*}
    \sigma_{{\rm ac}}(-\Delta_g)= [\frac{(n-1)^2}{4} K_0,\infty  ).
\end{equation*}
It yields (1) of Theorems \ref{Thembedfinite} and \ref{Thembedcountable}.


Also,  by  (\ref{controlkr}) and \eqref{Gapr95}, we have
  (3) of Theorems \ref{Thembedfinite}, \ref{Thembedcountable} hold.

By \eqref{controlkr} and (\ref{eigenor}), it suffices to show that for any $j$, $w_{\lambda_j}(x) \in L^2([1,\infty),dx)$.

Below we give the details.

For any $N(k_0-1)<j\leq N(k_0)$, by the construction (see (\ref{eigenj})), we have for $k\geq k_0$
\begin{eqnarray*}
   ||\left(\begin{array}{c}
      w_{\lambda_j}(J_{k+1}) \\
      \frac{1}{\sqrt{\lambda_j}}w_{\lambda_j}^{\prime}(J_{k+1})
    \end{array}
    \right)|| &\leq& 2^{N(k+1)+1}  N(k+1)^{100} C_{k+1}^{-99}||\left(\begin{array}{c}
      w_{\lambda_j}(J_{k}) \\
     \frac{1}{\sqrt{\lambda_j}} w_{\lambda_j}^{\prime}(J_{k}) \end{array}
    \right)|| \nonumber\\
    &\leq&    C_{k+1}^{-50}||\left(\begin{array}{c}
      w_{\lambda_j}(J_{k}) \\
     \frac{1}{\sqrt{\lambda_j}} w_{\lambda}^{\prime}(J_{k}) \end{array}
    \right)||, \label{lastiteration}
\end{eqnarray*}
where the second  inequality holds  by \eqref{cnnov}.

This implies for $k\geq k_0$
\begin{equation}\label{lastG}
    ||\left(\begin{array}{c}
      w_{\lambda_j}(J_{k+1}) \\
      \frac{1}{\sqrt{\lambda_j}}w_{\lambda_j}^{\prime}(J_{k+1})
    \end{array}
    \right)||\leq  T_{k_0}^{50}T_{k+1}^{-50}||\left(\begin{array}{c}
      w_{\lambda_j}(J_{k_0}) \\
     \frac{1}{\sqrt{\lambda_j}} w_{\lambda_j}^{\prime}(J_{k_0}) \end{array}
    \right)||.
\end{equation}
By (\ref{Key1}) and  (\ref{Controlallx1}), for all $x\in[J_{k+1},J_{k+2}]$,
\begin{eqnarray}
   ||\left(\begin{array}{c}
      w_{\lambda_j}(x) \\
      \frac{1}{\sqrt{\lambda_j}}w_{\lambda_j}^{\prime}(x)
    \end{array}
    \right)|| &\leq&  2^{N(k+2)}||\left(\begin{array}{c}
      w_{\lambda_j}(J_{k+1}) \\
     \frac{1}{\sqrt{\lambda_j}} w_{\lambda_j}^{\prime}(J_{k+1}) \end{array}
    \right)||  \nonumber\\
   &\leq& 2^{N(k+2)}  T_{k_0}^{50}T_{k+1}^{-50}||\left(\begin{array}{c}
      w_{\lambda_j}(J_{k_0}) \\
     \frac{1}{\sqrt{\lambda_j}} w_{\lambda_j}^{\prime}(J_{k_0}) \end{array}
    \right)|| \nonumber \\
     &\leq&
     T_{k_0}^{50}T_{k+1}^{- 49}||\left(\begin{array}{c}
      w_{\lambda_j}(J_{k_0}) \\
     \frac{1}{\sqrt{\lambda_j}} w_{\lambda_j}^{\prime}(J_{k_0}) \end{array}
    \right)||,\label{lastGx}
\end{eqnarray}
where the third inequality holds by \eqref{cnnov}.

Then by  (\ref{lastGx}), we have
\begin{eqnarray*}
  \int_{J_{k_0+1}}^{\infty} ||\left(\begin{array}{c}
      w_{\lambda_j}(x) \\
      \frac{1}{\sqrt{\lambda_j}}w_{\lambda_j}^{\prime}(x)
    \end{array}
    \right)|| ^2 dx&=&\sum_{k\geq k_0+1} \int_{J_{k}}^{J_{k+1}} ||\left(\begin{array}{c}
      w_{\lambda_j}(x) \\
      \frac{1}{\sqrt{\lambda_j}}w_{\lambda_j}^{\prime}(x)
    \end{array}
    \right)|| ^2 dx \\
   &\leq& T_{k_0}^{100}||\left(\begin{array}{c}
      w_{\lambda_j}(J_{k_0}) \\
     \frac{1}{\sqrt{\lambda_j}} w_{\lambda_j}^{\prime}(J_{k_0}) \end{array}
    \right)||^2 \sum_{k\geq k_0+1}\int_{J_{k}}^{J_{k+1}}  T_{k}^{-98} dx\\
   &\leq& T_{k_0}^{100}||\left(\begin{array}{c}
      w_{\lambda_j}(J_{k_0}) \\
     \frac{1}{\lambda_j} w_{\lambda_j}^{\prime}(J_{k_0}) \end{array}
    \right)||^2 \sum_{k\geq k_0+1} N(k+1)T_{k+1}T_{k}^{-98} \\
   &\leq&  T_{k_0}^{100} ||\left(\begin{array}{c}
      w_{\lambda_j}(J_{k_0}) \\
     \frac{1}{\sqrt{\lambda_j}} w_{\lambda_j}^{\prime}(J_{k_0}) \end{array}
    \right)||^2\sum_{k\geq k_0+1}N(k+1)C_{k+1}T_{k}^{-96} \\
     &\leq&  T_{k_0}^{100}||\left(\begin{array}{c}
      w_{\lambda_j}(J_{k_0}) \\
     \frac{1}{\sqrt{\lambda_j}} w_{\lambda_j}^{\prime}(J_{k_0}) \end{array}  \right)||^2 \sum_{k\geq k_0+1}T_{k}^{-90} <\infty.
\end{eqnarray*}
This completes the proof.
\end{proof}
 \section*{Acknowledgments}
We are grateful to A. Mramor for sharing with us the review \cite{donn}, which led
to the idea of this project. W.L. would like to thank M. Lukic and H. Xu for some useful discussions.  W.L. was
supported by the AMS-Simons Travel Grant 2016-2018. This research was
 supported by NSF DMS-1401204 and  NSF DMS-1700314.


\end{document}